\documentclass[runningheads]{llncs}
\usepackage{cite}
\usepackage{optidef}
\usepackage{amsmath,amssymb,amstext,amsfonts,graphics,enumerate}
\usepackage{xcolor}
\usepackage{hyperref}

\begin{document}

\title{Approximate Core Allocations for\\ Edge Cover Games}

\author{Tianhang Lu~\orcidID{0009-0001-8137-8149} \and
Han Xiao~\orcidID{0000-0002-7893-4028} \and
Qizhi Fang}

\authorrunning{Lu et al.}

\institute{Ocean University of China, Qingdao 266100, China\\
\email{tlu@stu.ouc.edu.cn},
\email{\{hxiao,qfang\}@ouc.edu.cn}
}

\maketitle

\begin{abstract}
We study the approximate core for edge cover games, which are cooperative games stemming from edge cover problems.
In these games, each player controls a vertex on a network $G=(V,E;w)$, and the cost of a coalition $S\subseteq V$ is equivalent to the minimum weight of edge covers in the subgraph induced by $S$.
We prove that the $\frac{3}{4}$-core of edge cover games is always non-empty and can be computed in polynomial time by using linear program duality approach.
This ratio is the best possible, as it represents the integrality gap of the natural LP for edge cover problems.
Moreover, our analysis reveals that the ratio of approximate core corresponds with the length of the shortest odd cycle of underlying graphs.
\keywords{Edge cover game \and approximate core \and linear program duality \and integrality gap.}
\end{abstract}

\section{Introduction} \label{intro}
Game theory studies the decision-making of rational, self-interested agents in strategic environments~\cite{OR94}.
Cooperative game theory is the branch of game theory which studies situations where players are able to making binding agreements about the distribution of payoffs outside the rules of the game~\cite{Pel07}.
One central problem in cooperative game theory is to distribute the total cost of cooperation to its participants.
There are many criteria for evaluating allocations~\cite{CEM11}~\cite{Pel07}, such as stability, fairness, and satisfaction.
Emphases on different criteria lead to different allocations, e.g., the core, the stable set, the Shapley value, the nucleon and the nucleolus.

The core~\cite{GD59}, which addresses the issue of stability, is one of the most attractive solution concepts in cooperative game theory.
The allocations in core are stable in the sense that no subset of players has an incentive to deviate from the grand coalition.
The approximate core, which is introduced by Aigle and Kern~\cite{Fai93}, provides an alternative solution for stability.
Unlike the core which can be empty, it offers an approximation to the core and is always existent, as long as the approximation ratio is bad enough.
Moreover, the approximate core captures a wider range of solution concepts compared to the core,
and it eventually reduces to the core when the approximation ratio equals one.
In the study of approximate core, a central problem is to determine the best ratio so that the approximation to the core is as close as possible.
This problem has been widely discussed for a number of cooperative games, such as matching games~\cite{Fai93}~\cite{Vaz21}~\cite{Xiao21}, TSP games~\cite{Pot92}~\cite{Fa98}~\cite{Sun15}, bin packing games~\cite{FK98}~\cite{KQ12}~\cite{Qiu16} and facility location games~\cite{Goe04}~\cite{KA83}.

Edge cover games was first studied by Deng et al.~\cite{DIN99} to model the cost allocation problem arising from edge cover problems.
From Gallai's Theorem, the core of edge cover game can be represented by the core of matching games.
In this sense, the core of the edge cover game may be empty by using of the fact that the core of the matching games may be empty.
On the other hand, the convex of the core can not be represented by the set of maximum independent sets on the underlying graph.

In another work, Liu and Fang~\cite{LF07} studied a variant of edge cover games and provided a complete characterization for the core and a sufficient condition for verifying the non-emptiness of the core.
In a follow-up work, Kim~\cite{KH16} studied rigid fractional edge cover games and its relaxed games.
They showed that a characterization of the cores of both games and found relationships between them.
Park et al.~\cite{PKK13} also studied different variants of edge cover games,
including rigid k-edge cover games and its relaxed games.
They gave a characterization of the cores of both games, found relationships between them,
and gave necessary and sufficient conditions for the balancedness of both of them.

In this work, we study edge cover games and present a characterization for the approximate core by employing the integrality gap of the underlying problem.
Our analysis demonstrates that the best ratio of the approximate core is upper bounded by the reciprocal of integrality gap.
Consequently, the most promising ratio for guaranteeing the non-emptiness of approximate core in edge cover games is $\frac{3}{4}$.
Additionally, we illustrate that it is always feasible to construct an allocation in the $\frac{3}{4}$-core of edge cover games efficiently.

The rest of this work is organized as follows.
In Section~\ref{sec:pre}, some notions and notations used in this paper are introduced.
Section~\ref{sec:main} is devoted to a characterization for the approximate core of edge cover games with the integrality gap.
Section~\ref{sec:conclusion} gives the concluding remark and some possible future research for the edge cover games.

\section{Preliminaries} \label{sec:pre}

A cooperative cost game $\Gamma (N,c)$ consists of a \emph{player set} $N=\{1,2,\ldots,n\}$ and a \emph{characteristic function} $c: 2^N \rightarrow \mathbb{R}$, where for each \textit{coalition} $S\subseteq N$, $c(S)$ represents the cost incurred by the players in $S$.
The \textit{core} of the game $\Gamma (N,c)$ is the set of vectors $a\in \mathbb{R}_+^{|N|}$ satisfying:
\begin{equation}
\label{eq:def-core}
  \begin{split}
  & a(S)\leq c(S) \text{~~for all~~} S \subseteq N, \\
  & a(N) = c(N).
  \end{split}
\end{equation}
where $a(S)= \sum\nolimits_{i\in S}a_i$.
We say a vector $a\in \mathbb{R}_+^{|N|}$ satisfies the \textit{core property} if $a(S)\leq c(S)$ for any $S \subseteq N$.
Given $0\leq \alpha \leq 1$, the \textit{$\alpha$-core} of the game $\Gamma (N,c)$ is the set of vectors $a\in \mathbb{R}_+^{|N|}$ satisfying:
\begin{equation}
\label{eq:def-approx-core}
  \begin{split}
  & a(S)\leq c(S) \text{~~for all~~} S \subseteq N, \\
  & a(N) \geq \alpha c(N).
  \end{split}
\end{equation}
A vector in the $\alpha$-core guarantees that no coalition will cost more than the cost it makes on its own, and the total cost of that allocation is at least $\alpha$ times of the cost of all players.
It is appealing to find the largest value $\alpha$ guaranteeing the $\alpha$-core being non-empty.
When the core is non-empty, the core is precisely the $\alpha$-core for $\alpha=1$.

Let $G=(V,E)$ be an undirected graph with vertices set $V$ and edges set $E$.
For any non-empty set $U\subseteq E$, the induced subgraph on $U$, denoted by $G[U]$, is a subgraph of $G$ with edges in $U$.
For any vertex subset $S\subseteq V$, $\delta(S)$ denotes the set of edges incident to exactly one vertex in $S$.
If $S$ contains a single vertex $v$, we use $\delta(\{v\})$ as an abbreviation for $\delta(v)$.

An \textit{edge cover} of $G$ is a set of edges $K \subseteq E$ such that $\delta(v) \cap K \neq \emptyset$ for any $v \in V$.
Given a non-negative weight function $w$ on $E$ that assigns a cost to each edge,
the \textit{minimum weight edge cover problem} aims to find an edge cover such that the total weight of edges is minimized.
The value of a minimum weight edge cover, denoted by $\gamma(G,w)$, is called the \emph{weighted edge cover number}.

Edge cover games study how to allocate the total cost of the edge cover among all players.
More precisely, $\Gamma_G(V,c)$ is the \textit{edge cover game} defined on an edge-weighted graph $G=(V,E;w)$.
The player set of $\Gamma_G(V,c)$ consists of the vertices in $V$.
For any coalition $S\subseteq V$,
the \textit{cost function} $c:2^V \rightarrow \mathbb{R}_+$ is defined by the minimum weight of edge set covering $S$.
In other words, $c(S)=\gamma(G[E[S]\cup \delta (S)],w)$ where $E[S]$ denotes set of edges that both endpoints are contained in $S$.

\section{The approximate core of edge cover games}\label{sec:main}

It is showed that the core of an edge cover game is non-empty if and only if there is no integrality gap for the underlying problem~\cite{DIN99}.
It turns out that the approximate core of edge cover games also admits a characterization with the integrality gap of the underlying problem.
Moreover, the largest ratio guaranteeing the approximate core being non-empty is upper bounded by the reciprocal of integrality gap.
Hence the problem of finding the largest ratio for the approximate core boils down to computing the integrality gap.
This section is threefold.
Subsection~\ref{sec:frac-edge-cover} shows how to compute an optimal half-integral edge cover.
Subsection~\ref{sec:inte-gap} utilizes the fractional edge cover computed in Subsection~\ref{sec:frac-edge-cover} to prove the integrality gap of edge cover problems.
Subsection~\ref{sec:prop-core} uses the integrality gap of edge cover problems to characterize the approximate core of edge cover games.

\subsection{Computing an optimal half-integral edge cover}\label{sec:frac-edge-cover}
To compute the integrality gap of edge cover problems, we resort to a class of special fractional edge covers, the optimal half-integral edge covers.
We show that an optimal half-integral edge cover can always be found efficiently.
Formally, a vector $x$ is called \textit{half-integral} if $2x$ is integral.

The following linear program~\eqref{lp:frac-covering} captures the minimum weight edge cover problem on $G=(V,E;w)$ by restricting variables to $0$ and $1$.

\begin{mini!}
	{}{\sum_{e\in E} w_e x_e}
  {\label{lp:frac-covering}}{}
	\addConstraint { \sum\limits_{e\in \delta(v)} x_e}{\geq 1 \quad}{v\in V \label{lp:frac-covering:c1}}
	\addConstraint { x_e }{\geq 0 \quad}{e\in E. \label{lp:frac-covering:c2}}
\end{mini!}

A feasible solution to LP~\eqref{lp:frac-covering} is called a \textit{fractional edge cover} in $G$.
An optimal solution to LP~\eqref{lp:frac-covering} is called a \textit{minimum fractional edge cover} in $G$.
In the case of bipartite graphs, the weight of minimum fractional edge cover has the same value as weighted edge cover number $\gamma(G,w)$.

\begin{lemma}[Schrijver~\cite{Sch03}] \label{lem:bip-half-integer}
If $G$ is a bipartite graph, then LP~\eqref{lp:frac-covering} has an integral optimal solution.
\end{lemma}

In the following, we show how to obtain an optimal fractional solution to LP~\eqref{lp:frac-covering} that is half-integral.
We employ the technique of edge doubling proposed by Nemhauser and Trotter~\cite{NT75},
whereby we create two copies of the vertex set $V$, denoted as $V^{\prime}$ and $V^{\prime \prime}$,
such that each vertex $v\in V$ corresponds to $v^{\prime}\in V^{\prime}$ and $v^{\prime\prime}\in V^{\prime\prime}$.
Next, we construct the graph $G'=(V'\cup V'',E')$, where $E'=\{u^{\prime}v^{\prime\prime}|uv\in E\}\cup \{u^{\prime\prime}v^{\prime}|uv\in E\}$,
and assign weight $w_{uv}$ to each of edges $u'v''\in E'$ and $u''v'\in E'$.
Since $G'$ is bipartite, it is possible to efficiently compute a minimum edge cover of $G'$, denoted as $F$.
We set $\overline{x}=1_F$, meaning that $\overline{x}_e=1$ if $e\in F$ and $\overline{x}_e=0$ otherwise.

Define $x^*$ from $\overline{x}$ by
\begin{equation}\label{eq:opt-half-integral}
x^*_{uv}=\frac{1}{2}\left(\overline{x}_{u'v''}+\overline{x}_{u''v'}\right) \text{~~for all~~} uv\in E.
\end{equation}
The following lemma shows that $x^*$ defined in~\eqref{eq:opt-half-integral} is an optimal half-integral edge cover of $G$.

\begin{lemma} \label{lem:half-integer-x}
$x^*$ is an optimal fractional edge cover of $G$.
\end{lemma}

\begin{proof}
We first show that $x^*$ is feasible for LP~\eqref{lp:frac-covering}.
For any vertex $v\in V$, we have
\begin{equation}
  \sum\nolimits_{uv \in \delta(v)} x^*_{uv}=\frac{1}{2} \sum\nolimits_{u''v' \in \delta(v')} \overline{x}_{u''v'}+\frac{1}{2} \sum\nolimits_{u'v'' \in \delta(v'')} \overline{x}_{u'v''} \geq 1,
\end{equation}
which implies the feasibility of $x^*$.
If $x^*$ is not an optimal half-integral edge cover in $G$,
there will be an optimal fractional edge cover $z$ in $G$ such that $\sum_{e \in E} w_ez_e<\sum_{e \in E} w_ex^*_e$.
Then, we define a feasible edge cover $\overline{z}$ in $G'$:
\begin{equation}
  \overline{z}_{u'v''}=\overline{z}_{u''v'}=z_{vu} \text{~~for all~~} uv\in E.
\end{equation}
This reduction get feasible of $\overline{z}$ and we have
\begin{equation}
\begin{split}
2\sum_{e \in E} w_ex^*_e&=\sum\nolimits_{u'v''\in E'}{w_{uv}\overline{x}_{u'v''}}+\sum\nolimits_{u''v'\in E'}{w_{uv}\overline{x}_{u''v'}}\\
&\leq \sum\nolimits_{u'v''\in E'}{w_{uv}\overline{z}_{u'v''}}+\sum\nolimits_{u''v'\in E'}{w_{uv}\overline{z}_{u''v'}}\\
&=2\sum_{e \in E} w_ez_e
\end{split}
\end{equation}
which contradicts with the assumption $\sum_{e \in E} w_ez_e<\sum_{e \in E} w_ex^*_e$.
\qed
\end{proof}

We can adjust the half-integral edge cover in the Lemma~\ref{lem:half-integer-x} so that the subgraph induced by fractional components of $x^*$ can be decomposed into vertex-disjoint odd cycles.
This can be achieved by rounding $x^*$ iteratively.

\begin{lemma} \label{lem:mod-half-integer-x}
There exists an optimal half-integral edge cover $\tilde{x}$ such the subgraph induced by fractional components of $\tilde{x}$ can be decomposed into vertex-disjoint odd cycles.
\end{lemma}

\begin{proof}
Initially, we set $\tilde{x}$ equal to $x^*$ and describe a procedure that generates another optimal solution with strictly more integer coordinates than $\tilde{x}$.
Let $H$ be the subgraph of $G$ induced by the set of edges $\{e\in E|\tilde{x}_e=\frac{1}{2}\}$.
First, we round $\tilde{x}$ to eliminate all of paths and even cycles in $H$.

Let $P=v_1v_2\ldots v_k$ be the longest path in $H$.
Note that if $e$ is an edge incident to $v_1$ and different from $v_1v_2$, then $x_e\neq \frac{1}{2}$;
otherwise, $H$ would contain either a cycle or a longer path.
Therefore, the edge $v_1v_2$ is the only edge connected to $v_1$ that has a half-integral value on $\tilde{x}$.
As $\tilde{x}$ is feasible for LP~\eqref{lp:frac-covering},
at least one edge incident to $v_1$ has a value of 1 in $\tilde{x}$.
Thus we have $\tilde{x}(\delta(v_1))\geq \frac{3}{2}$.
Likewise, we can deduce that $\tilde{x}(\delta(v_k))\geq \frac{3}{2}$.
Define $x^{\prime}$ and $x^{\prime\prime}$ as follows:
$$
x_e^{\prime}= \begin{cases}\tilde{x}_e-\frac{1}{2}, & \text { if } e=v_i v_{i+1}, 1 \leq i \leq k-1 \text { and } i \text { is odd, } \\ \tilde{x}_e+\frac{1}{2}, & \text { if } e=v_i v_{i+1}, 1 \leq i \leq k-1 \text { and } i \text { is even, } \\ \tilde{x}_e, & \text { if } e \notin E(P),\end{cases}
$$
and
$$
x_e^{\prime \prime}= \begin{cases}\tilde{x}_e+\frac{1}{2}, & \text { if } e=v_i v_{i+1}, 1 \leq i \leq k-1 \text { and } i \text { is odd, } \\ \tilde{x}_e-\frac{1}{2}, & \text { if } e=v_i v_{i+1}, 1 \leq i \leq k-1 \text { and } i \text { is even, } \\ \tilde{x}_e, & \text { if } e \notin E(P).\end{cases}
$$
There are two admissible solutions to LP~\eqref{lp:frac-covering}.
Moreover,
$$
\sum_{e \in E} w_e\tilde{x}_e=\frac{1}{2}\left(\sum_{e \in E} w_ex_e^{\prime}+\sum_{e \in E} w_ex_e^{\prime \prime}\right)
$$
Vectors $x'$ and $x''$ have integer coordinates in $P$, and share the same coordinates with $\tilde{x}$ on the other edges.
As $\tilde{x}$ is an optimal fractional edge cover, $x^{\prime}$ and $x^{\prime\prime}$ are also optimal solutions.

For any even cycle $C=v_1v_2\ldots v_k$ in $H$ with $v_1=v_k$,
we can use similar method to round $\tilde{x}$.
Define $x^{\prime}$ and $x^{\prime\prime}$ as follows:
$$
x_e^{\prime}= \begin{cases}\tilde{x}_e-\frac{1}{2}, & \text { if } e=v_i v_{i+1}, 1 \leq i \leq k-1 \text { and } i \text { is odd, } \\ \tilde{x}_e+\frac{1}{2}, & \text { if } e=v_i v_{i+1}, 1 \leq i \leq k-1 \text { and } i \text { is even, } \\ \tilde{x}_e, & \text { if } e \notin E(C),\end{cases}
$$
and
$$
x_e^{\prime \prime}= \begin{cases}\tilde{x}_e+\frac{1}{2}, & \text { if } e=v_i v_{i+1}, 1 \leq i \leq k-1 \text { and } i \text { is odd, } \\ \tilde{x}_e-\frac{1}{2}, & \text { if } e=v_i v_{i+1}, 1 \leq i \leq k-1 \text { and } i \text { is even, } \\ \tilde{x}_e, & \text { if } e \notin E(C).\end{cases}
$$
There are two admissible solutions to LP~\eqref{lp:frac-covering}.
Moreover,
$$
\sum_{e \in E} w_e\tilde{x}_e=\frac{1}{2}\left(\sum_{e \in E} w_ex_e^{\prime}+\sum_{e \in E} w_ex_e^{\prime \prime}\right)
$$
Thus $x^{\prime}$ and $x^{\prime\prime}$ are also optimal solutions, have integer coordinates in $C$ and share the same coordinates with $\tilde{x}$ on the other edges.

We continue this process until $H$ does not contain any path or even cycle.
Next, we proof that any two odd cycles in $H$ are vertex-disjoint.
If two odd cycles in $H$ are vertex-disjoint but not edge-disjoint, we can combine them into an even cycle and then round it.
Therefore, to prove that any two odd cycles are vertex-disjoint, it is sufficient to show that they are edge-disjoint.
Suppose that $H$ contains two cycles $C_1$ and $C_2$.
Let $P=v_1v_2\ldots v_k$ be the longest path belong to $C_1\cap C_2$.
We use the same method to obtain two vectors $x^{\prime}$ and $x^{\prime\prime}$ which have integer coordinates on $P$.
Then, we can replace $\tilde{x}$ by $x^{\prime}$ or $x^{\prime\prime}$ and continue this procedure until any two odd cycles in $H$ are edge-disjoint.
\qed
\end{proof}

\subsection{Integrality gap of edge cover problems} \label{sec:inte-gap}
This subsection studies the integrality gap of edge cover problems which will be used in characterizing the approximate core for edge cover games.
The \textit{edge cover polytope} of $G$, denoted by $\textrm{IP}(G)$, is the convex hull of incidence vectors of all edge covers of $G$.
The \textit{fractional edge cover polytope} of $G$, denoted by $\textrm{P}(G)$, is the convex hull of all fractional edge covers of $G$.
It follows that $\textrm{P}(G)$ is precisely the polytope defined by constraints~\eqref{lp:frac-covering:c1} and~\eqref{lp:frac-covering:c2}.
According to Edmonds~\cite{EJ70}, $\textrm{IP}(G)$ can be described by $\textrm{P}(G)$ after imposing the following odd set constraints:
\begin{equation} \label{eq:odd-set}
  x(E[U]\cup \delta(U))\geq \lceil \frac{1}{2}|U| \rceil \text{~~for all~~} U \subseteq V,~ |U| \text{ odd}.
\end{equation}

The \textit{integrality gap} of the edge cover problem on $G$, denoted by $\rho(G)$, is defined by
\begin{equation}\label{eq:integrality-gap}
  \rho(G)=\max\limits_{w:\mathbb{R}^{|E|} \rightarrow \mathbb{R}_+}\frac{\min\{wx:x\in \textrm{IP}(G)\}}{\min\{wx:x\in \textrm{P}(G)\}}.
\end{equation}
We have the following result for $\rho(G)$.
\begin{theorem} \label{thm:integrality-gap}
  Let $G=(V,E)$ be a graph.
  Then, $\rho (G) = 1 + \frac{1}{\ell(G)}$, where $\ell(G)$ is the length of the shortest odd cycle in $G$.
  Moreover, if $G$ is bipartite, then $\rho (G) = 1$.
\end{theorem}

\begin{proof}
We employ the result of Carr and Vempala~\cite{DV00}.
A \textit{dominant} $D(P)$ of a polyhedron $P\subseteq \mathbb{R}^n$ is the set of points $y\in \mathbb{R}^n$ which dominates some vector $x\in P$,
i.e., $D(P)=\{y\in \mathbb{R}^n:\exists x\in P, y\geq x\}$.

\begin{lemma}[Carr and Vempala~\cite{DV00}] \label{lem:gap-condition}
Given a polyhedron $P$ and its convex hull of the integer points $Z$, the integrality gap of the linear programming on $P$ is $r$ if and only if $r\geq 1$ is the smallest real number such that for any point $x^*$ of $P$, $rx^*\in D(Z)$.
\end{lemma}

If $G$ is bipartite, $\rho(G)=1$ follows from Lemma~\ref{lem:bip-half-integer} directly.
Hence we assume that $G$ is non-bipartite.

We first show that $\rho(G)\geq 1+\frac{1}{\ell(G)}$.
Let $C^*$ be a shortest odd cycle in $G$.
We obtain that
\begin{equation}
  \begin{split}
    \rho(G) & =\max\limits_{w:\mathbb{R}^{|E|} \rightarrow \mathbb{R}_+}\frac{\min\{wx:x\in \textrm{IP}(G)\}}{\min\{wx:x\in \textrm{P}(G)\}} \\
    & \geq \frac{\min\{1_{C^*}x:x\in \textrm{IP}(G)\}}{\min\{1_{C^*}x:x\in \textrm{P}(G)\}} \\
    & = \frac{(|C^*|+1)/2}{|C^*|/2} \\
    & = 1+\frac{1}{\ell(G)}.
  \end{split}
\end{equation}

Here the second-to-last equality holds because the minimum edge cover of $C^*$ can be attained by any matching in $C^*$ that exposes exactly one vertex,
while the minimum fractional edge cover of $C^*$ corresponds to an half-integral edge cover.

Then we show that $\rho(G)\leq 1+\frac{1}{\ell(G)}$.
Let $\tilde{x}$ denote the optimal half-integral edge cover of $G$ constructed in Lemma~\ref{lem:mod-half-integer-x}.
Lemma~\ref{lem:gap-condition} implies that the condition $\rho(G)\leq 1+\frac{1}{\ell(G)}$ holds if and only if $(1+\frac{1}{\ell(G)})\tilde{x}$ belongs to $\textrm{IP}(G)$.
Since that $(1+\frac{1}{\ell(G)})\tilde{x}$ is a feasible fractional edge cover, we only need to show that it satisfies~\eqref{eq:odd-set}.
Let $H_1$ and $H_2$ be the subgraph of $G$ induced by the fractional and integral components of $\tilde{x}$ respectively.
Then $H_1$ consists of vertex-disjoint cycles by Lemma~\ref{lem:mod-half-integer-x}.
Moreover, by picking alternate edges in each path with a length greater than 3, we can assume that $H_2$ is composed of vertex-disjoint stars.
Let $U$ be any odd set of vertices in $G$.
For any components $K$ of $G[E[U]\cup \delta(U)]$, there are four possible cases:
\begin{enumerate}
\item $K$ be a star in $H_2[E[U]]$, then $\tilde{x}(E[K])=|E[K]|$,
\item $K$ be a star in $H_2[\delta(U)]$, then $\tilde{x}(E[K])=|E[K]|$,
\item $K$ be a path in $H_1[E[U]\cup \delta(U)]$, then $\tilde{x}(K)=\frac{1}{2}|E[K]|$, and
\item $K$ be an odd cycle in $H_1[E[U]\cup \delta(U)]$, then $\tilde{x}(K)=\frac{1}{2}|E[K]|$.
\end{enumerate}
If there is no odd cycle in $G[E[U]\cup \delta(U)]$, then we have
$$
\begin{array}{lcl}
\left(1+\frac{1}{\ell(G)}\right)\tilde{x}(E[U]\cup \delta(U))&\geq & \tilde{x}(E[U]\cup \delta(U))\\[3pt]
&=& |E[U]|+|\delta(U)|\\[3pt]
& \geq & \lceil \frac{1}{2}|U| \rceil.
\end{array}
$$
The last inequality bases on the observation that $\tilde{x}(K)\geq \lceil \frac{1}{2}|V[K]| \rceil$ when $K$ falls in the first three cases, where $V[K]$ denotes the set of vertices of $K$.

Otherwise, $H_1$ contains an odd cycle, thus $\ell(G)\leq |U|$.
It follows that
$$
\begin{array}{lcl}
\left(1+\frac{1}{\ell(G)}\right)\tilde{x}(E[U]\cup \delta(U))& \geq & \frac{|U|+1}{|U|}\tilde{x}(E[U]\cup \delta(U))\\[3pt]
& \geq & \frac{|U|+1}{|U|} \cdot \frac{|U|}{2} \\[3pt]
&=&\lceil \frac{|U|}{2} \rceil.
\end{array}
$$
Therefore, we conclude that $(1+\frac{1}{\ell(G)})\tilde{x}$ belongs to $\textrm{IP}(G)$.
\qed
\end{proof}

\subsection{Characterizing approximate core with integrality gap} \label{sec:prop-core}
In this subsection, we introduce a characterization for the approximate core of edge cover games.
For any vertex $v$, $N(v)$ denotes the set of vertices adjacent to $v$.
If a vertex subset $S\subseteq N(v)$, $\delta(v,S)$ denotes the set of crossing edges between $v$ and $S$.
In additionally, the induced subgraph $G[\delta(v,S)]$ is called a \textit{star} or a \textit{$v$-star} with $v$ being the \textit{center}.

Liu and Fang~\cite{LF07} showed that the core of edge cover games is closely related to the stars in the underlying graph.
Based on the observation that any minimum edge cover can be partitioned into some vertex-disjoint stars,
we introduce the following characterizations for the core property.

\begin{lemma} \label{lem:cost-stru}
A vector $a\in \mathbb{R}_+^{|V|}$ satisfies the core property of $\Gamma_G(V,c)$ if and only if for any vertex $v\in V$ and vertex subset $T \subseteq N(v)$ the inequality $a(T \cup \{v\})\leq \sum\nolimits_{e\in \delta(v,T)}w_e$ holds.
\end{lemma}

\begin{proof}
Let $a\in \mathbb{R}_+^{|V|}$ be a vector satisfying the core property in $\Gamma_G(V,c)$, i.e.,
$a\geq 0$ and $a(S)\leq c(S)$ for any subset $S$ of $V$.
Since $T\subseteq N(v)$, $\delta(v,T)$ is an edge cover for $T \cup \{v\}$.
It follows that $a(T\cup \{v\}) \leq c(T \cup \{v\}) \leq \sum\nolimits_{e\in \delta(v,T)}w_e$.

To prove the converse, it suffices to show that  $a(S)\leq c(S)$ for any $S\subseteq V$.
Let $K$ denote the set of edges which covers $S$ with minimum weight.
It is evident that $K$ admits a star decomposition represented by $K_1,K_2,\ldots,K_l$.
For each $i=1,2,\ldots,l$, we define $u_i$ as the center of the star $K_i$, and $T_i$ as the set of vertices in the star $K_i$ except $u_i$.
Hence, we have $a(S)\leq \sum\nolimits_{i=1}^l a(T_i \cup \{u_i\}) \leq \sum\nolimits_{i=1}^l \sum\nolimits_{e\in \delta(u_i,S_i)}w_e=\sum\nolimits_{e\in K}w_e=c(S)$.
\qed
\end{proof}

Based on the linear programming formula of the $\alpha$-core,
we show that dual solutions of this game characterize the core property.

\begin{lemma} \label{lem:dual}
A vector $a\in \mathbb{R}_+^{|V|}$ satisfies the core property of $\Gamma_G(V,c)$ if and only if it is a dual feasible solution to the LP~\eqref{lp:frac-covering}.
\end{lemma}

\begin{proof}
Consider the dual of LP~\eqref{lp:frac-covering}.
\begin{maxi!}
{}{\sum\limits_{v \in V}y_v} 
{\label{lp:dual-covering}}{}
\addConstraint { y_u+y_v }{\leq w_{uv} \quad}{uv\in E}
\addConstraint {y_v}{\geq 0 \quad}{v\in V.}
\end{maxi!}

On the one hand, let $a\in \mathbb{R}_+^{|V|}$ be a vector satisfying core property of edge cover game $\Gamma_G(V,c)$.
By Lemma~\ref{lem:cost-stru}, it is easy to verify that $a_u+a_v\leq w_{uv}$ for any edge $uv \in E$.
This implies that $a$ is a feasible solution to LP~\eqref{lp:dual-covering}.

On the other hand, let $y$ be a feasible solution of LP~\eqref{lp:dual-covering}, we show that $y$ satisfies the core property.
Let $v\in V$ and $T\subseteq N(v)$.
Since $y_u+y_v\leq w_{uv}$, it follows that $y(T \cup \{v\}) \leq y(T)+|T|y_v \leq \sum\nolimits_{e\in \delta(v,T)}w_{uv}$.
By Lemma~\ref{lem:cost-stru}, $a$ satisfies the core property.
\qed
\end{proof}

Now we are ready to characterize the approximate ratio in terms of the integrality gap.

\begin{theorem} \label{thm:appcore}
Let $\Gamma_G(V,c)$ be the edge cover game defined on non-bipartite graph $G=(V,E;w)$.
Then the $\frac{\ell(G)}{1+\ell(G)}$-core of $\Gamma_G(V,c)$ is always non-empty and can be computed efficiently.
Moreover, $\frac{\ell(G)}{1+\ell(G)}$ is the largest ratio guaranteeing the non-emptiness of the approximate core of $\Gamma_G(V,c)$.
\end{theorem}

\begin{proof}
An optimal solution $a^*$ to the dual of LP~\eqref{lp:frac-covering} can be computed in polynomial time using standard linear programming techniques.
By Lemma~\ref{lem:dual}, $a^*$ satisfies the core property, i.e., $a^*(S)\leq c(S)$ for any coalition $S\subseteq V$.
Besides, we have
$$
\begin{array}{lcl}
\rho(G) a^*(V) &=& \rho(G) \cdot \min\{wx:x\in \textrm{P}(G)\}\\[3pt]
& \geq & \min\{wx:x\in \textrm{IP}(G)\}\\[3pt]
&=& c(V).
\end{array}
$$
According to the definition of the approximate core, $a^*$ is a $\frac{\ell(G)}{1+\ell(G)}$-core for $\Gamma_G(V,c)$.
Algorithms for finding a shortest odd cycle of a graph can be finished in time $O(|V||E|)$ by using breadth-first search in~\cite{IR78}.
Thus we can calculate $\rho(G)$ in polynomial time of $|V|$ and $|E|$.

Now we show that $\frac{\ell(G)}{1+\ell(G)}$ is the largest ratio guaranteeing the non-emptiness of the approximate core.
Suppose $a$ is a vector in the $\alpha$-core of $\Gamma_G(V,c)$.
By the definition, the $a$ satisfies the core property.
Thus $a$ is a feasible dual solution to LP~\eqref{lp:frac-covering} by Lemma~\ref{lem:dual}.
We have
\begin{equation*}
  \alpha \gamma(G,w)= \alpha c(V) \leq a(V) \leq \min\{wx:x\in \textrm{P}(G)\},
\end{equation*}
where the last inequality follows by the weak duality theorem of linear programming.
It follows that
\begin{equation*}
\alpha \leq \frac{\min\{wx:x\in \textrm{P}(G)\}}{\gamma(G,w)}.
\end{equation*}
By the definition of integrality gap, we have $\alpha \leq \frac{1}{\ell(G)}$.
\qed
\end{proof}

Since the length of the shortest odd cycle is at least $3$, 
the ratio will degenerate to $\frac{3}{4}$ if there is any triangle in the graph.
\begin{corollary}
Let $\Gamma_G(V,c)$ be the edge cover game defined on graph $G=(V,E;w)$.
Then the $\frac{3}{4}$-core of $\Gamma_G(V,c)$ is always non-empty.
Moreover, an allocation in the $\frac{3}{4}$-core of $\Gamma_G(V,c)$ can be computed efficiently.
\end{corollary}

\section{Conclusion} \label{sec:conclusion}
In this paper, we considered a cost allocation problem for the edge cover game. 
We characterize approximate core by using the dual solution of the nature linear programming problem.
Therefore, the best approximate factor depends on the integrality gap between the integer linear programming problem and its relaxation.
To estimate this factor, we employ linear programming rounding techniques and prove that it is $1+\frac{1}{\ell(G)}$.
This result ensures that the $\frac{\ell(G)}{1+\ell(G)}$-core of the edge cover game is always non-empty.
Additionally, when the shortest odd cycle in the underlying graph is equal to 1,
our proposed solution degenerates into a factor of $\frac{3}{4}$.

One possible working direction for edge cover games is to study the allocation of nucleon,
where Faigle et al.~\cite{FKF98} studied the nucleon of matching games and Kern and Paulusma~\cite{KP09} studied the nucleon of simple flow games.
Our result might be helpful since a nucleon locals in the allocation of the largest satisfaction ratio.
Besides, variants of edge cover games introduced by Liu and Fang~\cite{LF07} are also worth studying.

\subsubsection{Acknowledgements.}
The work is supported in part by in part by the National Natural Science Foundation of China (Nos.\,12001507, 11871442, 11971447 and 12171444) and Natural Science Foundation of Shandong (No.\,ZR2020QA024).

\bibliographystyle{splncs04}
\bibliography{ref.bib}
\nocite{*}

\end{document}